\documentclass[12pt,a4paper]{article}
\usepackage[a4paper, left = 0.7in, right = 0.7in, top = 1in, bottom = 1in]{geometry}
\usepackage{amsthm,amsmath,amssymb}
\usepackage{latexsym}
\usepackage{hyperref}
\usepackage{comment}

\theoremstyle{plain}
\newtheorem{theorem}{Theorem}[section]
\newtheorem{lemma}[theorem]{Lemma}
\newtheorem{proposition}[theorem]{Proposition}
\newtheorem{corollary}[theorem]{Corollary}

\theoremstyle{definition}
\newtheorem{definition}[theorem]{Definition}

\theoremstyle{remark}
\newtheorem{remark}[theorem]{Remark}

\newcommand{\norm}[1]{\left\Vert#1\right\Vert}
\newcommand{\vabs}[1]{\left\vert#1\right\vert}
\newcommand{\set}[1]{\left\{#1\right\}}

\newcommand{\R}{\mathbb{R}}
\newcommand{\N}{\mathbb{N}}
\newcommand{\C}{\mathbb{C}}
\newcommand{\Z}{\mathbb{Z}}
\newcommand{\Q}{\mathbb{Q}}

\numberwithin{equation}{section}

\newcommand{\abs}[1]{\lvert#1\rvert}
\newcommand{\sprod}[2]{\left<#1,#2\right>}
\DeclareMathOperator{\supp}{supp}
\DeclareMathOperator{\lyap}{Lyap}

\usepackage{tikz}
\usetikzlibrary{calc,trees,positioning,arrows,chains,shapes.geometric,%
    decorations.pathreplacing,decorations.pathmorphing,shapes,%
    matrix,shapes.symbols}
\usetikzlibrary{arrows,patterns}

\definecolor{Blue}{rgb}{0.1, 0.1, 0.5}
\definecolor{DBlue}{rgb}{0.1, 0.1, 0.3}
\definecolor{Green}{rgb}{0.1, 0.55, 0.1}
\definecolor{DGreen}{rgb}{0.1, 0.35, 0.1}

\newcount\colveccount
\newcommand*\colvec[1]{
        \global\colveccount#1
        \begin{pmatrix}
        \colvecnext
}
\def\colvecnext#1{
        #1
        \global\advance\colveccount-1
        \ifnum\colveccount>0
                \\
                \expandafter\colvecnext
        \else
                \end{pmatrix}
        \fi
}

\DeclareMathOperator{\Tr}{Tr}

\begin{document}

\title{Transport exponents of states with large support}

\author{Vitalii Gerbuz \thanks{The author was supported in part by NSF grant DMS-1361625.}}
\date{}

\maketitle

\abstract{We extend results of Damanik and Tcheremchantsev on estimating transport exponents to initial states supported on more than one site. These general results for upper and lower bounds are then applied to several classes of models, including Sturmian, quasi-periodic and substitution-generated potentials, and the random polymer model.}

\section{Introduction}

We study the quantum evolution under the action of a discrete Hamiltonian $H$ acting on~$l^2(\Z)$ 
:
\begin{equation}\label{schrodinger_equation}
(H\psi)_n = \psi_{n+1}+ \psi_{n-1} + W(n)\psi_n
\end{equation}
where $W:\Z\rightarrow\R$ is called \textit{potential}.

The wave function $\psi :\Z\rightarrow\C$, also called a \textit{quantum state} or a \textit{wave packet}, describes the quantum state of the particle in the sense that $|\psi_n|$ is the probability distribution of the position of the particle in the medium modeled by the potential. The evolution of this distribution from the  initial state $\psi = \psi(0)$  is given by $\psi(t) = e^{-itH}\psi(0)$ (\cite{teschlQM}).  We look at the time-averaged entries of $\psi(t)$:\footnote{Some authors denote time averaged quantities by tilde, e.g. \cite{DGLQ14}. We will not since no non-averaged quantities are considered.}
\begin{equation*}
a(\psi,n,T) = \frac{2}{T}\int\limits_0^\infty e^{-2t/T}|\psi(t,n)|^2 dt
\end{equation*}
It is of interest to study the so-called \textit{transport exponents}, which capture the asymptotic behavior in time of a wave packet. There are several methods that allow estimates on transport exponents. Some of them work for any properly localized initial state. However some methods and papers only treat $\psi = \delta_0$, i.e. a state initially localized at one spot state, and very little is known for initial states with larger support.

In this paper we will prove lower bounds on transport exponents for a certain classes of operators and compactly supported initial states and upper bounds for states with infinite support and exponentially decaying entries. We apply these results to a broad class of examples, in particular, Sturmian potentials, H\"older continuous quasi-periodic potentials, Hamiltonians with potentials generated by period doubling and Thue-Morse substitutions and the random polymer model.

In order to define transport exponents, one needs to introduce the right and left time-averaged \textit{outside probabilities}
\begin{align*}
&P_r(\psi,N,T) = \sum\limits_{n\geq N} a(\psi,n,T)\\
&P_l(\psi,N,T) = \sum\limits_{n\leq -N} a(\psi,n,T)
\end{align*}
The sum of the above is called the time-averaged outside probability
\begin{equation*}
P(\psi,N,T) = P_r(\psi,N,T) + P_l(\psi,N,T)
\end{equation*}
The time evolution of the wavepacket can be captured through the following quantities:
\begin{equation}\label{def_S_minus}
S^-(\psi,\alpha) = -\liminf\limits_{T\rightarrow\infty}\frac{\log P(T^\alpha,T)}{\log T}
\end{equation}
\begin{equation}\label{def_S_plus}
S^+(\psi,\alpha) = -\limsup\limits_{T\rightarrow\infty}\frac{\log P(T^\alpha,T)}{\log T}
\end{equation}
We now define the \textit{upper transport exponents}
\begin{equation*}
\alpha_u^\pm(\psi) = \sup\set{\alpha>0: S^\pm(\psi,\alpha) < \infty}
\end{equation*}
It follows from this definition that if $P(T^\alpha,T)$ goes to $0$ faster then any inverse power of $T$, then $\alpha^+_u \leq \alpha$, and if it is only true for a sequence of times $T_k$, then we get $\alpha^-_u \leq \alpha$.

A slightly different approach to the same dynamical quantities is through the moments of the position operator:
\begin{equation*}
\langle|X_\psi|^p\rangle(T) = \sum\limits_{n\in\Z}|n|^pa(\psi,n,T)
\end{equation*}
For each $p>0$ we define a pair of transport exponents:
\begin{equation*}
\beta^+(\psi,p) = \limsup\limits_{T\rightarrow\infty}\frac{\log \langle|X_\psi|^p\rangle(T)}{p\log T}; \;
\beta^-(\psi,p) = \liminf\limits_{T\rightarrow\infty}\frac{\log \langle|X_\psi|^p\rangle(T)}{p\log T}
\end{equation*}
It was shown in \cite{GKTch04} that $\beta^\pm(\psi,p)$ is a monotone increasing function of $p$ and
\begin{equation}\label{beta_limits_to_alpha}
\lim\limits_{p\rightarrow\infty}\beta^\pm(\psi,p) = \alpha^\pm_u(\psi)
\end{equation}

There are two main ideas to estimating transport exponents. One of them relates transport to the continuity properties of the spectral measures while the other only exploits the rates of transfer matrices growth.
The first approach was initially developed by Guarneri, Combes and Last in \cite{Combes1993}, \cite{Guarneri89}, \cite{LAST96}. For a large class of operators and any initial state $\psi$ they proved the following bound:
\begin{equation}\label{hd_lower_bound}
\beta^-(p) \geq \dim_H\mu_\psi
\end{equation}
where $\dim_H\mu_\psi$ is the Hausdorff dimension of the spectral measure corresponding to $\psi$.
For applications to particular models the work of Jitomirskaya and Last \cite{JL99} provided a tool to estimate Hausdorff dimensions through the growth rates of solutions to the eigenvalue equation via the subordinacy theory. For example for Sturmian potentials with bounded density frequency this was done in \cite{DKL00}.
The benefit of this method is its ability to handle arbitrary initial states. However it is often the case that \ref{hd_lower_bound} gives an estimate that is far from sharp. For instance the random polymer model may exhibit pure-point (and  hence zero-dimentional spectrum) and almost ballistic transport at the same time. This model will be discussed in section \ref{random polymer model}.

We will focus on the second approach developed by Damanik and Tcheremchantsev in \cite{DT03},\cite{DT07},\cite{DT08}. They showed that upper bounds on the norms of transfer matrices lead to lower bounds for transport as well as lower bounds on the transfer matrices lead to upper bounds on the transport exponents. They applied this method to several models including Fibonacci Hamiltonian, Thue-Morse, Period-Doubling and quasi-periodic potentials generated by circle rotation with trigonometric polynomial sampling function. Recently Jitomirskaya and Mavi \cite{JM16} used the same method to prove vanishing of transport exponents for quasi-periodic potentials generated by piecewise-H\"older function. This broad class of models, in particular, includes the almost Mathew operator -- one of the most prominent Schr\"odinger operators. This approach was initially developed and applied to the initial state localized at one spot. We extend several of these results to the case of exponentially localized initial states.

\section{The Main Results}

The main results of this paper, upper and lower bounds on transport exponents under fairly general asumptions, are given in Theorem~\ref{thm_upper_bound_transport_via_tr_m} and Theorem~\ref{main_result}.

However to develop some ideas of working with nonsingular states let us start with a fairly obvious statement saying that the sum of two wavepackets cannot propagate faster then any of them individually.

Fix operator $H$ of the form \ref{schrodinger_equation}.
\begin{proposition}\label{exponent_sublinearity}
Let $\psi_1, \psi_2 \in L^2(\Z)$ such that $\beta^+(\psi_1,p), \beta^+(\psi_2,p)$ exist and
\begin{align*}
 &\beta^+({\psi_1},p) \leq \beta_1\\
 &\beta^+({\psi_2},p) \leq \beta_2,
\end{align*}
then for any $\psi = x\psi_1+ y\psi_2,\,x,y\in\R$
\begin{equation}
\beta^+(\psi,p)\leq \max\set{\beta_1, \beta_2}.
\end{equation}
\end{proposition}
The proof of this proposition will be given in Section~\ref{section_upper_bounds}. For now we will just state a corollary.
\begin{corollary}
The same is true for $\alpha^+$.
\begin{proof}
Follows from \ref{beta_limits_to_alpha}.
\end{proof}
\end{corollary}
An upper bound on the $\beta^\pm$  transport exponents of the state initially localized at one site was proved in \cite{DT07}. For a wide class of operators the $\beta^+$ estimate can be extended to finitely supported initial states according to this proposition.
However it is not clear whether a similar result holds for $\beta^-(p)$. So we provide a corollary of \cite[Theorem 7]{DT07} which is a key statement leading to estimates on both $\alpha^+$ and $\alpha^-$.


We need a notion of the transfer matrices associated with a given Schr\"odinger operator~\ref{schrodinger_equation}.

\textit{Notation:} for any sequence $u\in\C^\infty$ we will write capital letter for a corresponding sequence of vectors:
\begin{equation*}
U(n) = \colvec{2}{u_{n+1}}{u_n}
\end{equation*}
For $n,k\in\Z,\, z\in\C$ define the transfer matrix
\begin{equation*}
M(n,k,z) =
\begin{cases}
A(n,z)\cdots A(k+1,z) &n > k,\\
Id &n=k,\\
A^{-1}(n+1,z)\cdots A^{-1}(k,z), &n < k
\end{cases}.
\end{equation*}
where
\begin{equation*}
A(n,z) = \begin{pmatrix}
z-W(n) & -1\\
1 & 0
\end{pmatrix}.
\end{equation*}
Thus $u:\Z\rightarrow\C$ solves $Hu_z = z u_z$ if and only if
\begin{equation*}
U_z(n) = M(n,k,z) U_z(k) \quad \forall n,k\in\Z.
\end{equation*}

Our first result gives an upper bound for the upper transport exponents. In Section~\ref{section_application} this theorem and ideas of its proof will be applied to Sturmian and piecewise H\"older continuous potentials.

\begin{definition}
A state $\psi\in\ell^2(\Z), \psi = \sum\psi_i\delta_i$ is called \textit{exponentially decaying} if there are constants $D,a > 0$ such that $|\psi(n)| \leq De^{-a|n|}.$
\end{definition}

\begin{theorem}\label{thm_upper_bound_transport_via_tr_m}
Let $H$ be a Schr\"odinger operator with bounded potential $V$ and $K>4$ such that $\sigma(H)\subseteq[-K+1, K-1]$ and let $\psi$ be exponentially decaying. Assume that for some $C>0, \alpha\in(0,1)$, as $T\rightarrow\infty$
\begin{equation}\label{as1}
\max\limits_{-\frac{CT^\alpha}{2}< k < \frac{CT^\alpha}{2}} \int\limits_{-K}^K \left(\max\limits_{k\leq n \leq CT^\alpha}\norm{M(n,k,E+\frac{i}{T})} \right)^{-2}dE = O(T^{-m})
\end{equation}
and
\begin{equation}\label{as2}
\max\limits_{-\frac{CT^\alpha}{2}< k < \frac{CT^\alpha}{2}} \int\limits_{-K}^K \left(\max\limits_{-CT^\alpha\leq n \leq k}\norm{M(n,k,E+\frac{i}{T})} \right)^{-2}dE = O(T^{-m})
\end{equation}

for every $m > 1$.
Then
\begin{equation*}
\alpha^+_u(\psi) \leq \alpha.
\end{equation*}
If the conditions are satisfied for some sequence of times $T_l\rightarrow\infty$, then $\alpha^-_u(\psi) \leq \alpha$.
\end{theorem}

Assumptions \ref{as1}, \ref{as2} can be rewritten in a simpler form in the case when the potential $W$ in \ref{schrodinger_equation} is dynamically defined. Such potentials make up a large class to which many of popular models belong. For example Sturmian potentials, Almost Mathieu operator and general quasi-periodic operators.

 Recall that any dynamical system $(\Omega,T), T:\Omega\rightarrow \Omega$ together with a sampling function~$f:\Omega\rightarrow \R$ defines a family of Schr\"odinger operators $H(\omega)$ given by \ref{schrodinger_equation} with potentials
\begin{equation*}
W(n,\omega) = f(T^n(\omega)), \, \omega\in \Omega.
\end{equation*}
We will assume that $T$ is a bijection on $\Omega$.
\begin{corollary}\label{cor_dyn_upper_bound}
Given dynamically defined family of Schr\"odinger operators $H(\omega)$ assume for some $K>4$, $\sigma(H(\omega))\subseteq[-K+1, K-1]$ for all $\omega\in \Omega$, $\psi$  is exponentially decaying. Assume for some $C>0, \alpha\in(0,1)$ and for every $m$,, as $T\rightarrow\infty$
\begin{equation}\label{as_dyn}
T^m \int\limits_{-K}^K \left(\max\limits_{0\leq n \leq \frac{CT^\alpha}{2}}  \norm{M(n,0,E+\frac{i}{T},\omega)} \right)^{-2} dE \rightarrow 0
\end{equation}
uniformly in $\omega$.
Then the conclusion of the theorem \ref{thm_upper_bound_transport_via_tr_m} holds for every $H(\omega)$.
\end{corollary}

Now we will discuss results concerning lower bounds on the transport exponents of initial states with compact support.

In \cite{DT03} Damanik and Tcheremchantsev proved dynamical lower bounds on the spreading of the initially localized wavepacket given a bound on the growth rate of the norms of transfer matrices. A similar result was proved for continuum Schr\"odinger operators in \cite{DLS06} by Damanik, Lenz and Stolz. The key difference from \cite{DT03} was that the estimate was proved for arbitrary initial state $f$ with compact support. However this extension required an additional assumption in the theorem: $\left<f,u\right> \neq 0$ for at least one solution $u$ of $Hu=E_0u$.

It turns out that a similar theorem can be proved in the discrete case:

\begin{theorem}\label{main_result}
Let $supp(\psi)\subset[-s,s]$; $\varepsilon_0,C,K,\alpha\in(0,\infty)$. Assume for every $N\in\N$ there exists a closed Borel set $A(N)\subset [-K,K]$ such that
\begin{equation}\label{transfer_matrix_bound}
\norm{M(n,m;E)}\leq CN^\alpha\quad \forall n,m: |n|,|m|<N, \forall E\in A(N)
\end{equation}
and for every $E\in A(N)$ there exists a formal solution $v_E$ to $Hv = Ev$ normalized by $\norm{V_E(0)} = 1$ such that
\begin{equation}\label{positive_product_condition}
|\sprod{\psi}{v_E}|>\varepsilon_0.
\end{equation}
Let $N(T) = T^{1/(1+\alpha)}$ and let $B(T)$ be the $1/T$ neighborhood of $A(N(T))$:
\begin{equation*}
B(T) = \set{E\in\R: \exists E'\in A(N), |E-E'|<\frac{1}{T}}.
\end{equation*}
Then for some $\hat C>0$ and all $T$ sufficiently large the following holds:
\begin{equation}\label{outside_probability_estimate}
P(\psi,\frac{N(T)}{2},T) \geq \frac{\hat C}{T}|B(T)|N^{1-2\alpha}(T)
\end{equation}
In particular we have a bound for the moments of the position operator:
\begin{equation}\label{moments_bound}
\langle|X|^p_\psi \rangle(T) \geq \frac{\hat C}{2^p} |B(T)|T^\frac{p-3\alpha}{1+\alpha}
\end{equation}
\end{theorem}
\begin{remark}
The only difference from the \cite[Theorem 1]{DT03} is the assunption \ref{positive_product_condition}. It is not clear whether this condition is actually necessary for the result to hold.
\end{remark}

\begin{corollary}\label{cor_one_energy_finite}
Let $supp(\psi)\subset [-s,s]$, $E_0\in \R$,
\begin{equation}\label{one_energy_condition}
\norm{M(n,m;E_0)}\leq C(|n|+|m|)^\alpha\; \forall n,m\in\Z
\end{equation}
and
\begin{equation}\label{one_energy_nonzero_cond}
\sprod{\psi}{v_{E_0}} \neq 0\text{ for some solution of }Hv = E_0v
\end{equation}
then
\begin{equation}\label{pth_exponent_estimate}
\beta^-(\psi,p) \geq \frac{1}{1+\alpha}-\frac{1+4\alpha}{p(1+\alpha)}
\end{equation}
In particular,
\begin{equation*}
\alpha^-(\psi) \geq \frac{1}{1+\alpha}
\end{equation*}
\begin{proof}
In Theorem~\ref{main_result} take $A(N) = \set{E_0}$. Then $|B(T)| = \frac{2}{T}$, so the result follows from \ref{moments_bound}.
\end{proof}
\end{corollary}

Theorem~\ref{thm_upper_bound_transport_via_tr_m} and Corollary~\ref{cor_dyn_upper_bound} will be proved in Section~\ref{section_upper_bounds}, Theorem~\ref{main_result} in Section~\ref{section_lower_bounds}. In the last section results are applied to various models, including Fibbonacci and general Sturmian Hamiltonians, Hamiltonians generated by Thue-Morse and Period-Doubling substitutions, H\"older quiasi-periodic potentials and the random polymer model.

\section{Upper Bounds}\label{section_upper_bounds}

\begin{proof}[Proof of Proposition \ref{exponent_sublinearity}]
Recall the definition of $\beta^+(p):$
\begin{equation*}
\beta_\psi^+(p) = \limsup\limits_{T\rightarrow \infty}\frac{1}{p\log T}\log\left( \frac{2}{T}\int\limits_0^\infty \left[e^{-\frac{2t}{T}}\sum\limits_{n\in\Z}|n|^p |\sprod{exp(-itH)\psi}{\delta_n}|^2\right]dt\right).
\end{equation*}
Since
\begin{align*}
|\sprod{exp(-itH)\psi}{\delta_n}|^2 = |x\sprod{exp(-itH)\psi_1}{\delta_n} + y\sprod{exp(-itH)\psi_2}{\delta_n}|^2\leq \\
\leq x^2 a(\psi_1,n,t) + y^2a(\psi_2,n,t),
\end{align*}
\begin{equation*}
\log(a+b) \leq \log 2 + \log\max\set{a,b}\quad\text{ for } a,b > 1
\end{equation*}
and
$$\limsup\limits_{T\rightarrow \infty}\max(f(T),g(T)) \leq \max(\limsup\limits_{T\rightarrow \infty}f(T),\limsup\limits_{T\rightarrow \infty}g(T))$$
 the desired follows.
\end{proof}

The proof of the theorem \ref{thm_upper_bound_transport_via_tr_m} is based on the following
\begin{proposition}\label{prop_upper_bound_general}

Let $H$ be a Schr\"odinger operator with bounded potential $V$ and $K>4$ such that $\sigma(H)\subseteq[-K+1, K-1]$, $\psi\in\ell^2(\Z), \psi = \sum\psi_i\delta_i$. Let $M\leq N$ be a natural number. Then the outside probabilities can be bounded as follows:

\begin{multline}\label{outside_right_inf}
P_r(\psi,N,T) \lesssim exp(-cN) +\\+ T^3\sum\limits_{|k|<M}|\psi_k|^2\int\limits_{-K}^K \left( \max\limits_{k\leq n \leq N}{\norm{M(n,k;E+\frac{i}{T})}}\right)^{-2}dE + \sum\limits_{|k|\geq M}|\psi_k|^2
\end{multline}
\begin{multline}\label{outside_left_inf}
P_l(\psi,N,T) \lesssim exp(-cN) +\\+ T^3\sum\limits_{|k|<M}|\psi_k|^2\int\limits_{-K}^K \left( \max\limits_{-N\leq n \leq k}{\norm{M(n,k;E+\frac{i}{T})}}\right)^{-2}dE + \sum\limits_{|k|\geq M}|\psi_k|^2.
\end{multline}
where the implicit constants depend only on K.
\end{proposition}
\begin{proof}
Since the Schr\"odinger evolution of $\delta_k$ is equivalent to evolution of $\delta_0$ under the Hamiltonian with potential shifted by $-k$, one gets the following from \cite[Theorem 7]{DT07}:
\begin{align*}
\sum\limits_{n\geq N} a(\delta_k,n,T) \lesssim exp(-cN) + T^3\int_{-K}^K \left(\max\limits_{k\leq n \leq N}{\norm{M(n,k;E+\frac{i}{T})}}\right)^{-2} dE \\
\sum\limits_{n\leq -N} a(\delta_k,n,T) \lesssim exp(-cN) + T^3\int_{-K}^K \left(\max\limits_{-N\leq n \leq k}{\norm{M(k,n;E+\frac{i}{T})}}\right)^{-2} dE
\end{align*}
where the implicit constant depends only on $K$.
\newline
Following the similar steps as in the proposition \ref{exponent_sublinearity}, we get for $\psi = \sum \psi_i \delta_i$:
\begin{equation*}
a(\psi,n,T)\leq \sum\limits_{k\in supp(\psi)} |\psi_i|^2 a(\delta_i, n,T).
\end{equation*}
Now summing over sites with $|k|<M$ and $|k|\geq M$, we get the desired estimate.
\newline
\end{proof}

\begin{proof}[Proof of the Theorem \ref{thm_upper_bound_transport_via_tr_m}]
Fix $m\in\N$.
In the assumptions of Proposition \ref{prop_upper_bound_general} set $N = CT^\alpha, M = \frac{CT^\alpha}{2}$, then \ref{outside_right_inf} becomes
\begin{multline}
P_r(\psi,CT^\alpha,T) \lesssim exp(-cCT^\alpha) +\\+ T^3\sum\limits_{|k|<\frac{CT^\alpha}{2}}|\psi_k|^2\int\limits_{-K}^K \left( \max\limits_{k\leq n \leq CT^\alpha}{\norm{M(n,k;E+\frac{i}{T})}}\right)^{-2}dE + \sum\limits_{|k|\geq \frac{CT^\alpha}{2}}|\psi_k|^2\leq \\
\leq exp(-cCT^\alpha) +\\
+ T^3\left(\sum\limits_{|k|<\frac{CT^\alpha}{2}}|\psi_k|^2\right) \max\limits_{-\frac{CT^\alpha}{2}< k < \frac{CT^\alpha}{2}} \int\limits_{-K}^K \left(\max\limits_{k\leq n \leq CT^\alpha}\norm{M(n,k,E+\frac{i}{T})} \right)^{-2}dE +\\
+\sum\limits_{|k|\geq \frac{CT^\alpha}{2}}|\psi_k|^2 \lesssim T^{-m}
\end{multline}
In the last step we used the assumption of the theorem and the exponential decay of $\psi$.

The same estimate holds for $P_l(\psi,CT^\alpha,T)$. To plug it into the definition \ref{def_S_plus} of $S^+(\psi,\alpha)$ we need to work with $P(\psi,T^\alpha,T)$. Choose $\varepsilon > 0$, then for $T>T_\varepsilon: P(\psi,T^{\alpha+\varepsilon},T) \leq P(\psi,CT^\alpha,T)$ and
\begin{equation}\label{eq_upper_bounds}
S^+(\psi,\alpha+\varepsilon) \geq -\limsup\limits_{T\rightarrow\infty}\frac{\log T^{-m}}{\log T} \geq m.
\end{equation}
Since this is true for any $m\geq1$, $S^+(\psi,\alpha+\varepsilon) = \infty$ and $\alpha_u^+(\psi) \leq \alpha+\varepsilon$. Since $\varepsilon>0$  was arbitrary, $\alpha_u^+(\psi) \leq \alpha$.

If \ref{as1},\ref{as2} are satisfied only for a sequence of times then taking liminf in \ref{eq_upper_bounds} along that subsequence gives $\alpha_u^-(\psi) \leq \alpha$.

\end{proof}

\begin{proof}[Proof of Corollary \ref{cor_dyn_upper_bound}]
We only need to show how \ref{as_dyn} implies \ref{as1} and \ref{as2}.

Note that for $|k|<\frac{CT^\alpha}{2}$
\begin{align*}
\max\limits_{k\leq n \leq CT^\alpha}{\norm{M(n,k;E+\frac{i}{T},\omega)}} = \max\limits_{0\leq l \leq CT^\alpha-k}{\norm{M(l,0;E+\frac{i}{T},\omega+k\theta)}} \geq\\
 \geq \max\limits_{0\leq l \leq \frac{CT^\alpha}{2}}{\norm{M(l,0;E+\frac{i}{T},\omega+k\theta)}}.
\end{align*}
From that one easily sees that \ref{as_dyn} implies \ref{as1} due to uniformity of the limit in $\omega$.

To get \ref{as2} note that for $n<k$
\begin{equation*}
\norm{M(n,k;z,\omega)} = \norm{M(k,n;z,\omega)^{-1}} =\norm{M(k,n;z,\omega)} = \norm{M(k-n,0;z,\omega+n\theta)}
\end{equation*}
where the second equality is due to $\det M(k,n;z) = 1$.

Now again for $|k|<\frac{CT^\alpha}{2}$
\begin{align*}
\max\limits_{-CT^\alpha\leq n \leq k}{\norm{M(n,k;E+\frac{i}{T},\omega)}} = \max\limits_{-CT^\alpha\leq n \leq k}{\norm{M(k-n,0;E+\frac{i}{T},\omega+n\theta)}} = \\
=\max\limits_{0\leq l \leq CT^\alpha+k}{\norm{M(l,0;E+\frac{i}{T},\omega+(k-l)\theta)}} \geq \max\limits_{0\leq l \leq \frac{CT^\alpha}{2}}{\norm{M(l,0;E+\frac{i}{T},\omega+k\theta)}}.
\end{align*}
So \ref{as2} can be derived from \ref{as_dyn} too.
\end{proof}

In Section~\ref{section_application} we will apply Corollary~\ref{cor_dyn_upper_bound} to several models from the class of operators with quasi-peiodic potentials.

\section{Lower Bounds}\label{section_lower_bounds}
The restriction of a sequence to the interval $[-s,s]$ will be denoted by $u_{|s}$, the restriction to the interval $I = [a,b]$ by $u_{|I}$.

\begin{lemma}\label{lemma_liminf}
Let $\psi\in l^2(\Z)$, $E_0\in\R$, $I=[a,b]$, $a,b\in\Z$. For any $z\in\C\backslash \R$, let $u_z = (H-z)^{-1}\psi$ and suppose that
\begin{align}\label{liminf_condition}
\lim\limits_{\delta\rightarrow 0}\inf\set{|u_z(b+1)|^2+|u_z(b)|^2+|u_z(a)|^2+|u_z(a-1)|^2: z\in\C\backslash\R, |z-E_0|<\delta} < \epsilon.
\end{align}
Then for any solution $v_{E_0}$ of $Hv = E_0v$,
\begin{equation*}
|\sprod{v_{E_0|I}}{\psi_{|I}}| < \sqrt{\varepsilon}(\norm{V_{E_0}(b)}+\norm{V_{E_0}(a-1)}).
\end{equation*}
\begin{proof}
Let \newline $l = \lim\limits_{\delta\rightarrow 0}\inf\set{|u_z(b+1)|^2+|u_z(b)|^2+|u_z(a)|^2+|u_z(a-1)|^2: z\in\C\backslash\R, |z-E_0|<\delta}$.
Then we can pick $z_n\rightarrow E_0$, $z_n\in\C\backslash\R$, such that
\begin{equation}\label{liminf_condition_seq}
\liminf\limits_{n\rightarrow\infty}\left(|u_{z_n}(b+1)|^2+|u_{z_n}(b)|^2+|u_{z_n}(a)|^2+|u_{z_n}(a-1)|^2\right) = l.
\end{equation}
Since $\set{U_{z_n}(a-1)}_{n\in\N}$ is a bounded sequence in $\C^2$, we can choose a convergent subsequence. Abusing the notation, denote it by the same index sequence $z_n$.
Let $u_{E_0}$ be the unique formal solution to $Hu = E_0u+\psi$ with initial conditions
\begin{align*}
u_{E_0}(a-1) &=  \lim\limits_{n\rightarrow\infty}u_{z_n}(a-1)\\
u_{E_0}(a) &=  \lim\limits_{n\rightarrow\infty}u_{z_n}(a)
\end{align*}
For any $j,\,u_z(j)$ continuously depends on $z$ and the initial conditions $U_z(a)$. Therefore
\begin{equation*}
u_{z_n}(j)\rightarrow z_{E_0}(j) \text{ for any } j\in\Z
\end{equation*}
In particular,
\begin{equation*}
|u_{E_0}(b+1)|^2+|u_{E_0}(b)|^2+|u_{E_0}(a)|^2+|u_{E_0}(a-1)|^2 = l < \varepsilon
\end{equation*}
Since $u_{E_0}$ solves difference equation,
\begin{multline}\label{s_restriction_inhomog}
(Hu_{E_0})_{|I} = (E_0u_{E_0} + \psi)_{|I} \Leftrightarrow\\
H_{|I}u_{E_0|I} + u_{E_0}(b+1)\delta_b + u_{E_0}(a-1)\delta_{a} = E_0u_{E_0|I} + \psi_{|I} \Leftrightarrow \\
 \psi_{|I} = (H_{|I}-E_0)u_{E_0|I} + u_{E_0}(b+1)\delta_b + u_{E_0}(a-1)\delta_{a} .
\end{multline}
Let $v_{E_0}$ formally solve $Hv_{E_0} = E_0v_{E_0}$, then
\begin{align*}
(Hv_{E_0})_{|I} = E_0v_{E_0|I}\Leftrightarrow (H_{|I}-E_0)v_{E_0|I} = -v_{E_0}(b+1)\delta_b-v_{E_0}(a-1)\delta_{a}.
\end{align*}
Take the dot product with $v_{E_0|I}$ on both sides of \ref{s_restriction_inhomog}:
\begin{align*}
\sprod{\psi_{|I}}{v_{E_0|I}} = \sprod{(H_{|I}-E_0)u_{E_0|I}}{v_{E_0|I}} + u_{E_0}(b+1)v_{E_0}(b) + u_{E_0}(a-1)v_{E_0}(a) = \\
= \sprod{u_{E_0|I}}{(H_{|I}-E_0)v_{E_0|I}} + u_{E_0}(b+1)v_{E_0}(b) + u_{E_0}(a-1)v_{E_0}(a) = \\
= \sprod{u_{E_0|I}}{-v_{E_0}(b+1)\delta_b - v_{E_0}(a-1)\delta_{a}} + u_{E_0}(b+1)v_{E_0}(b) + u_{E_0}(a-1)v_{E_0}(a) = \\
= -v_{E_0}(b+1)u_{E_0}(b) - v_{E_0}(a-1)u_{E_0}(a) + u_{E_0}(b+1)v_{E_0}(b) + u_{E_0}(a-1)v_{E_0}(a)
\end{align*}
Finally using the Cauchy inequality and $|u_{E_0}(b)|^2 + |u_{E_0}(b+1)|^2 < \varepsilon$, $|u_{E_0}(a)|^2 + |u_{E_0}(a-1)|^2 < \varepsilon$ we can estimate
\begin{align*}
|\sprod{\psi_{|I}}{v_{E_0|I}}| \leq \sqrt{\varepsilon}\left(\sqrt{|v_{E_0}(b)|^2 + |v_{E_0}(b+1)|^2} + \sqrt{|v_{E_0}(a)|^2 + |v_{E_0}(a-1)|^2}\right) =\\
= \sqrt{\varepsilon}\left(\norm{V_{E_0}(b)}+\norm{V_{E_0}(a-1)}\right)
\end{align*}

\end{proof}
\end{lemma}

\begin{proof}[Proof of Theorem \ref{main_result}]
We use the strategy from \cite[Theorem 1]{DT03}.
Using Parseval's identity one gets:
\begin{equation}\label{parseval2}
a(\psi,n,T) = \frac{2}{T}\int\limits_0^\infty e^{-2t/T}|\sprod{e^{-itH}\psi}{\delta_n}|^2 dt = \frac{1}{\pi T}\int\limits_\R|\sprod{(H-E-\frac{i}{T})^{-1}\psi}{\delta_n}|^2dE
\end{equation}
Since $\psi$ is supported on $[-s,s]$,
\begin{equation*}
U_z(n) = M(n, s+1; z)U_z(s+1) \text{ for any } n>s+1,
\end{equation*}
\begin{equation*}
U_z(n) = M(n, -s-1; z)U_z(-s-1) \text{ for any } n<-s-1,
\end{equation*}
and $det(M(n,m; z)) = 1$, we get the following key inequalities for any $z\in\C$:
\begin{align}\label{key_inequality2}
&\norm{U_z(n)} \geq \norm{M(n, s+1; z)}^{-1}\norm{U_z(s+1)}  & &\text{for any } n>s+1,\\
&\norm{U_z(n)} \geq \norm{M(n, -s-1; z)}^{-1}\norm{U_z(-s-1)}& &\text{for any } n<-s-1.
\end{align}
Choose $N = T^\frac{1}{1+\alpha}$, $\epsilon = \frac{1}{T}$, $z = E+i\epsilon$ with $|E-E_0|<\epsilon$.
Using \ref{transfer_matrix_bound} and \cite[Lemma 2.1]{DT03} we get the following estimate on the norm of the transfer matrix at any energy $z: |z-E'|<\epsilon, E'\in A(N(T))$:
\begin{equation*}
\norm{M(n,s+1;z)}\leq DN^\alpha \text{  for } s+1<n<N \text{ and }
\end{equation*}
\begin{equation*}
\norm{M(n,-s-1;z)}\leq DN^\alpha \text{  for } -N<n<-s-1
\end{equation*}
with some universal constant $D$ that depend only on $C$.
\newline Now \ref{key_inequality2} can be rewritten as
\begin{align}
&|u_z(n+1)|^2+|u_z(n)|^2 \geq D^2 N^{-2\alpha}\norm{U_z(s+1)}^2 &&\text{ for any } n>s+1,\\
&|u_z(n+1)|^2+|u_z(n)|^2 \geq D^2 N^{-2\alpha}\norm{U_z(-s-1)}^2 &&\text{ for any } n<-s-1.
\end{align}
Finally we estimate the sum
\begin{multline}\label{outside_sum_estimate}
\sum\limits_{|n|>N/2}(|u_z(n+1)|^2+|u_z(n)|^2) \geq \sum\limits_{n>N/2}^N + \sum\limits_{n>-N}^{-N/2} \\
\geq C_1\frac{N}{2}N^{-2\alpha}(\norm{U_z(s+1)}^2 + \norm{U_z(-s-1)}^2)
\end{multline}
Since for every fixed $n,m$, the entries of $M(n,m,E)$ are polynomials in $E$, there is a constant $C_s$ such that for $E\in[-K,K]$
\begin{align*}
&\norm{M(s,0;E)} < C_s \\
&\norm{M(1,-s-1;E)} < C_s .
\end{align*}
Since \ref{positive_product_condition} holds for every $E\in A(N)$, we can apply Lemma \ref{lemma_liminf} to get:
\begin{align*}
\lim\limits_{\delta\rightarrow 0}\inf\set{\abs{U_z(s)}^2+\abs{U_z(-s-1)}^2: z\in\C\backslash\R, |z-E|<\delta}>\\
>\frac{\varepsilon_0^2}{\left(\norm{V_E(s)}+\norm{V_E(-s-1)}\right)^2} \geq \frac{\varepsilon_0^2}{2\left(\norm{V_E(s)}^2+\norm{V_E(-s-1)}^2\right)}\geq\\
[\text{ using } \norm{V_E(0)} = 1]\geq \qquad\frac{\varepsilon_0^2}{4C_s^2} =: 2\varepsilon_1
\end{align*}
\textit{Remark: }The factor 2 is introduced here for the convenience of further calculations.
\newline For each $E\in A(N(T))$ define $\delta(E)>0$ such that
\begin{equation*}
\inf\set{\abs{U_z(s)}^2+\abs{U_z(-s-1)}^2: z\in\C\backslash\R, |z-E|<\delta(E)} > \varepsilon_1
\end{equation*}
Since $A(N)$ is a compact set, one can choose a finite subcover from the cover \\ $\set{\left(E-\delta(E),E+\delta(E)\right)}_{E\in A(N)}$, denote it by $\set{\left(E_i-\delta(E_i),E_i+\delta(E_i)\right)}_{i\in[1,M]}$. Let \\$\hat\delta = \min{\delta(E_i)}$.
Now we have that for any $z\in\C\backslash\R,\, dist(z,A(N))<\hat\delta$
\begin{equation*}
\abs{U_z(s)}^2+\abs{U_z(-s-1)}^2 > \varepsilon_1
\end{equation*}
Thus one can pick $T$ large enough so that $\frac{1}{T}<\frac{\hat\delta}{2}$. Then it follows from \ref{outside_sum_estimate}
\begin{equation*}
\sum\limits_{|n|>N/2}(|u_z(n)|^2) \geq C_1\varepsilon_1/2 N^{1-2\alpha} =: C_2 N^{1-2\alpha}
\end{equation*}
and
\begin{align}\label{sum_estimates}
\int\limits_\R\sum\limits_{|n|>N/2}\vabs{\sprod{(H-E-\frac{i}{T})^{-1}\psi}{\delta_n}}^2dE \geq \int\limits_{B(T)} C_2N^{1-2\alpha} dt = \\ = C_2 \abs{B(T)}N^{1-2\alpha}
\end{align}
Summation over $|n|>N(T)/2$ in \ref{parseval2} with \ref{sum_estimates} proves \ref{outside_probability_estimate} with $\hat C = \frac{C_2}{\pi}$.

We also get an estimate on the p-th moment of the position operator:
\begin{equation*}
\left<\vabs{X}^p_\psi\right>(T) \geq \left(\frac{N}{2}\right)^p\sum\limits_{|n|>N/2}a(\psi,n,T) \geq \frac{\hat C}{2^p}\frac{|B(T)|}{T}N^{p+1-2\alpha} = \frac{\hat C}{2^p} |B(T)|T^\frac{p-3\alpha}{1+\alpha}
\end{equation*}
which proves \ref{moments_bound}
\end{proof}

In a lot of cases estimates like \ref{transfer_matrix_bound} are known for a rich set of energies. However verifying $\sprod{\psi}{v_{E_0}} \neq 0$ requires addressing.

In some cases it can even be eliminated. Similar to Corollary \ref{cor_one_energy_finite} we prove

\begin{proposition}\label{prop_one_energy_inf_many}
Let $supp(\psi)\subset [-s,s]$, $E_0\in \R$, suppose also that \ref{one_energy_condition} is satisfied for an infinite set of energies $E_0$, then
\begin{equation*}
\beta^-_{\psi}(p) \geq \frac{1}{1+\alpha}-\frac{1+4\alpha}{p(1+\alpha)}.
\end{equation*}
\end{proposition}
\begin{proof}
Fix some initial conditions $(v_E(0),v_E(1))$ for a solution to $Hv_E = Ev_E$. Then $v_E(n)$ is a polynomial of degree $|n|$ in $E$. Consequently $\sprod{\psi}{v_{E_0}}$ is a polynomial of positive degree (the degree is the largest index of a non-zero entry of $\psi$), so it has a finite number of zeroes. Therefore since \ref{one_energy_condition} is satisfied for an infinite number of values of $E_0$, $\sprod{\psi}{v_{E_0}}\neq 0$ for at least one of them, so \ref{cor_one_energy_finite} can be applied.
\end{proof}

\section{Applications}\label{section_application}

In this Section we will apply results of Sections~\ref{section_upper_bounds} and \ref{section_lower_bounds} to several important models: Fibonacci Hamiltonian and general Sturmian potentials, "rough" quasi-periodic potentials, period doubling and Thue-Morse substitution generated potentials and the random polymer model.

\subsection {Fibonacci Hamiltonian}
The operator of the form \ref{schrodinger_equation} with potential
\begin{equation}\label{fibbonacci_potential}
W(n) = \lambda ([(n+1)\theta+\omega]-[n\theta+\omega]),\quad \theta\in[0,1]\backslash\Q,\, \omega\in[0,1),\, \lambda>0,
\end{equation}
where $\theta = \frac{\sqrt{5}-1}{2}$ is called the Fibonacci Hamiltonian.

From now on fix some $\lambda > 0$.
In order to apply Corollary~\ref{cor_dyn_upper_bound} one needs to prove lower bounds on the growth of transfer matrices norms. This is usually done via studying the so-called trace map (see e.g. \cite{DT08}).
First we define the \textit{renormalized} transfer matrices $M(k,z,\omega) = M(q_k,0,z,\omega)$, where $q_k = F_k$ is the $k^{th}$ Fibonacci number. Then consider $x_k(z,\omega) = \Tr M(k,z,\omega)$. For now let $\omega = 0$. $x_k(z)$ is a polynomial of degree $q_k$. We also introduce
\begin{equation*}
\sigma_k(\delta) = \set{z\in\C:\; |x_k(z)| \leq 1+\delta}
\end{equation*}
It was shown in \cite{DGY} that for small enough $\delta = \delta(\lambda)$, $x_k(z)$ has precisely $q_k$ zeroes $\set{E_k^{(j)}}$ all lying on the real line and $\sigma_k(\delta) = \bigcup\limits_{j = 1}^{q_k} \sigma_k^{(j)}(\delta)$ consists of $q_k$ disjoint connected components in $\C$ with exactly one zero in each component.

The sizes of these connected components are important in estimating growth rates of transfer matrices. To quantify them we introduce
\begin{align}\label{inner_radii}
&r_k^{(j)} (\delta) = \sup\set{r>0: B_r(E_k^{(j)})\subseteq \sigma_k^{\delta}   },\quad &&r_k(\delta) = \max\limits_{j\in\set{1,\ldots,q_k}}\set{r_k^{(j)}(\delta)}\\\label{outer_radii}
&R_k^{(j)} (\delta) = \inf\set{R>0: B_R(E_k^{(j)})\supseteq \sigma_k^\delta   },\quad &&R_k(\delta) = \max\limits_{j\in\set{1,\ldots,q_k}}\set{R_k^{(j)}(\delta)}
\end{align}

\begin{proposition}\label{thm_fibb_upper}
Fix $\lambda>0, \omega\in[0,1)$.
Let $\psi$  be exponentially decaying.
Then
\begin{equation}\label{fibb_exponent_upper_1}
\alpha^+_u(\psi) \leq \frac{\log\varphi}{\liminf\limits_{k\rightarrow\infty}\frac{1}{k}\log\frac{1}{R_k(\delta)}}
\end{equation}
where $\varphi = (1+\sqrt{5})/2 = \theta^{-1}$.
\begin{proof}
With suitable $\rho>0$ set
\begin{equation*}
s = \frac{\liminf\limits_{k\rightarrow\infty}\frac{1}{k}\log\frac{1}{R_k(\delta)}}{\log\varphi}-\rho>0
\end{equation*}
Then there exists $C_\delta>0$ such that $R_k(\delta) < C_\delta F_k^{-s}$.

Fix some $\delta>0, T>1$. Denote by $k = k(T)$ the unique integer satisfying
\begin{equation}\label{k(T)}
\frac{F_{k(T)-1}^s}{C_\delta} \leq T < \frac{F_{k(T)}^s}{C_\delta}.
\end{equation}
Let $N(T) = F_{k(T)+\sqrt{k(T)}}$. Note that for any $\nu > 0$, there is a constant $C_\nu > 0$ such that
\begin{equation}\label{N(T)}
N(T) \leq C_\nu T^{\frac{1}{s}}T^\nu.
\end{equation}
It was shown in \cite[Proposition 3.8]{DGY} that for any $E, \omega$\footnote{Estimates on the transfer matrices growth are usually first proved for $\omega = 0$ and then extended to all $\omega$'s by the observation that the trace of the word of length $F_n$ is the same for all but one word and if it happens that we need to estimate that exact word, we can shift it by one symbol to get a good word and then correct the estimate by a constant. That constant wouldn't affect the asymptotic of growth. }
\begin{equation}\label{fibb_exp_growth}
\norm{M(N(T),0,E+i/T,\omega)} \geq (1+\delta)^{F_{\sqrt{k(T)}}}.
\end{equation}
It is known that $F_n = \lfloor \theta^{-n}/\sqrt{5}\rfloor$, in particular, $\frac{1}{3}\theta^{-n} < F_n < \theta^{-n}$. Then from \ref{k(T)} it follows, that
\begin{equation*}
C_\delta^{1/s} T^{1/s} < F_{k(T)} < \theta^{-k(T)} \Rightarrow k(T) > C_1 \log T
\end{equation*}
Therefore $\norm{M(N(T),0,E+i/T,\omega)}$ decays faster then any power of $T$ uniformly in $E$ and $\omega$, which implies by Corollary \ref{cor_dyn_upper_bound}
\begin{equation*}
\alpha_u^+ \leq \frac{1}{s}+\nu = \left( \frac{\liminf\limits_{k\rightarrow\infty}\frac{1}{k}\log\frac{1}{R_k(\delta)}}{\log\varphi}-\rho\right)^{-1}+\nu.
\end{equation*}
Since $\rho$ and $\nu$ were arbitrary, the proof is complete.
\end{proof}
\end{proposition}
The trace map $Tr: \R^3\rightarrow\R^3$ is defined in such a way that it takes $(x_{k}(z),x_{k-1}(z),x_{k-2}(z))^T$ to $(x_{k+1}(z),x_{k}(z),x_{k-1}(z))^T$ (see \cite{DGY}).
It easy to check that $Tr$ preserves the following family of surfaces
\begin{equation*}
S_\lambda = \set{(x,y,z)\in\R^3:\, x^2 + y^2 +z^2 -2xyz = 1 + \frac{\lambda^2}{4}}
\end{equation*}
Cantat proved in \cite{cantat09} that for any $\lambda > 0$ there is a locally maximal compact hyperbolic set $\Lambda_\lambda\subset S_\lambda$ of $Tr$ with a dense set of periodic points in it.
There is a reformulation of \ref{fibb_exponent_upper_1} in terms of this map.
\begin{theorem} For the Fibonacci Hamiltonian with any coupling $\lambda > 0$ and any exponentially decaying state $\psi$
\begin{equation*}
\alpha^+_u(\psi) \leq \frac{\log\varphi}{\inf\limits_{p\in Per(\Lambda_\lambda)}Lyap^u(p)}
\end{equation*}
where $Lyap^u$ is the upper Lyapunov exponent for the trace map associated with the Fibonacci potential. For details see Section~\ref{section_upper_bounds}.

\begin{proof}
It was shown in \cite{DGY} that
\begin{equation}\label{R_k bound}
\frac{1}{R_k(\delta)} \geq \frac{\delta^2}{(1+\delta)^2(1+2\delta)^2} \cdot\min\limits_{1\leq j\leq q_k}|x_k'(E_k^{(j)})|
\end{equation}
and
\begin{equation}\label{inf Lyap}
\lim\limits_{k\rightarrow\infty}\frac{1}{k}\log\min\limits_{1\leq j \leq q_k}\abs{x_k'(E_k^{(j)})} = \inf\limits_{p\in Per(\Lambda_\lambda)} Lyap^u(p)>0
\end{equation}

Combining these with \ref{fibb_exponent_upper_1} we get the desired inequality.
\end{proof}
\end{theorem}

Now we will apply the results of Section~\ref{section_lower_bounds} to the Fibonacci Hamiltonian \ref{fibbonacci_potential}.
Following the ideas of \cite{DT03} we obtain the following lower bound for transport exponents.
\begin{proposition}
Denote $\eta =  2\frac{\log[\sqrt{5+2\lambda}(3+\lambda)a_\lambda]}{\log\varphi}$, where $a_\lambda$ is the largest root of\\ $x^3-(2+\lambda)x-1$. Then for every $\lambda > 0, p>0$ and every finitely-supported state $\psi$
\begin{equation*}
\beta^-(\psi,p) \geq \frac{1}{1+\eta}-\frac{1+4\eta}{p(1+\eta)}
\end{equation*}
\begin{proof}
It was shown in \cite{DGY} (see proof of Proposition 3.8) that for any $E\in \sigma_k(\delta), \omega\in[0,1), 1\leq n \leq q_k$
\begin{equation*}
\norm{M(n,0,E,\omega)} \leq C n^{\eta}.
\end{equation*}
It follows that for any $E_0\in\Sigma_\lambda$, which is covered by infinitely many $\sigma_k$'s we have
\begin{equation*}
\norm{M(n,m,E_0,\omega)} = \norm{M(n-m,0,E_0,\omega+m\theta)} \leq C |n-m|^{\eta} < C (|n|+|m|)^{\eta}
\end{equation*}
for any $n,m \geq 0$. For the negative values of $n,m$ one uses the symmetry of the potential for $\omega = 0: W(-n+1) = W(n),n>0.$ For $\omega \neq 0$ the argument from the footnote on the previous page works.

Thus we get that \ref{one_energy_condition} is satisfied for infinitely many points, therefore Proposition~\ref{prop_one_energy_inf_many} can be applied.

\end{proof}
\end{proposition}
\begin{corollary}
 For any finitely supported initial state $\psi$
\begin{equation*}
\alpha^-_u(\psi) \geq \frac{1}{1+\eta}.
\end{equation*}
\end{corollary}

Using the trace map analysis from \cite{DGY} plus some work to fulfill condition~\ref{positive_product_condition}, yet another bound for the Fibonacci Hamiltonian transport exponent can be found.

\begin{theorem}
Let $\lambda>0, \omega\in[0,1), \psi$ has finite support. Then
\begin{equation*}
\alpha_u^- = \alpha_u^+ =  \frac{\log\varphi}{\inf\limits_{p\in Per(T_\lambda)} \lyap^u(p)}
\end{equation*}
\end{theorem}

\begin{remark}
Since we already showed the upper bound for $\alpha_u^+$, what is left is to prove is
\begin{equation}\label{fibb_lower_bound}
\alpha_u^- \geq \frac{\log\varphi}{\inf\limits_{p\in Per(T_\lambda)} \lyap^u(p)}
\end{equation}
It will be shown in the Proposition \ref{prop_fibb_lower_rk}.
\end{remark}

Before we prove this estimate -- an analogue of \cite[Proposition 3.8]{DGY} -- we need to recall and refine some properties of the trace map. In view of Theorem~\ref{main_result}, the main obstacle for applying the original proof of \ref{fibb_lower_bound} to the finite support case is that we need to exclude from consideration energies that do not satisfy the positive product condition \ref{positive_product_condition}. In particular, zeroes of $x_k$ with largest $r_k^{(j)}$ can fall into that category. To deal with exceptional energies we prove a local in energy result (an analogue of \cite[Proposition 3.4]{DGY}) connecting $x_K'(E_k)$ to the expanding rates of the trace map along unstable manifolds.

\begin{lemma}\label{lem_derivative_Tr_1}
Let $\lambda > 0$, let $U = (a,b)$ be an interval such that $U \cap \Sigma_\lambda\neq \emptyset$. Then for any $\varepsilon > 0$, there exists $k_0\in\N$, such that for every $k>k_0$, $p\in\Lambda_\lambda$ there is $E_k\in U$ satisfying $x_k(E_k) = 0$ and
\begin{equation*}
\frac{1}{k}\log\norm{DTr^k(p)|_{E_p^u}} - \varepsilon \leq \frac{1}{k}\log |x_k'(E_k)| \leq \frac{1}{k}\log\norm{DTr^k(p)|_{E_p^u}} + \varepsilon
\end{equation*}
\end{lemma}
\begin{proof}
Let $U_\lambda$ be the image of $U$ on the line of initial conditions $\ell_\lambda$.

In order to use the proof of \cite[Proposition 3.4]{DGY} we need to justify the following statement: fix a small $\delta > 0$, then there exists $k'\in\Z_+$ such that $Tr^{-k'}(W^s_\delta(p)) \cap U_\lambda \neq \emptyset$ for any $p\in\Lambda_\lambda$.

Since $\Lambda_\lambda$ is a locally maximal topologically mixing totally disconnected hyperbolic set, the stable manifold of any point is dense in the stable lamination on $S_\lambda$.
Since $\tilde\Sigma_\lambda = \ell_\lambda\cap W^s$, where $\tilde\Sigma_\lambda$ is the image of the spectrum on the line of initial conditions, $W^s$ is the stable lamination, it follows that every stable manifold will intersect $U_\lambda$.

Also one can pick a Markov partition with arbitrary small elements, in particular smaller then $\delta$. After a finite number of iterates of $Tr^{-1}$ each element of the Markov partition will intersect every other element. Moreover every element being stretched along stable direction will eventually intersect $U_\lambda$. Thus after sufficiently many iterations every $\delta-$stable manifold will intersect $U_\lambda$, which proves the claim.

The rest of the proof is identical to \cite[Proposition 3.4]{DGY}.
\end{proof}

The following lemma is \cite[lemma 3.5]{DGY}
\begin{lemma}\label{lem_derivative_Tr_2}
Let $\lambda > 0$. For every $\varepsilon > 0$, there exists $k_0$ such that for any $k>k_0$ and every $E_k$ with $x_k(E_k) = 0$, one can find $p\in\Lambda_\lambda$ such that
\begin{equation*}
\frac{1}{k}\log\norm{DTr^k(p)|_{E_p^u}} - \varepsilon \leq \frac{1}{k}\log |x_k'(E_k)| \leq \frac{1}{k}\log\norm{DTr^k(p)|_{E_p^u}} + \varepsilon
\end{equation*}
\end{lemma}
\begin{proposition}\label{prop_fibb_lower_rk}
Let $\lambda>0, \omega\in[0,1)$ and $\psi$ has finite support. Then
\begin{equation}\label{fibb_exp_lower_1}
\alpha_u^-(\psi) \geq \frac{\log\varphi}{\inf\limits_{p\in Per(T_\lambda)} \lyap^u(p)}
\end{equation}
\end{proposition}

\begin{proof}
Denote $u_z = (H-E-\frac{i}{T})^{-1}\psi\in\ell^2(\Z), z = E+i/T$. As in the proof of Theorem~ \ref{main_result} we begin with
\begin{equation}\label{outside_probability_estimate_fibb}
P(N,T) = \frac{1}{\pi T}\sum_{|n|>N}\int\limits_\R |u_z(n)|^2dE \geq \frac{1}{2\pi T}\int\limits_{\sigma_k\cup\sigma_{k+1}}\sum_{|n|>N}^{2N} \norm{M(n,s+1,z,\omega)}^{-2}\norm{U_z(s+1)}^2 dE
\end{equation}
We will estimate the norms of transfer matrices using the trace map approach. For $\norm{U_z}$ we will apply Lemma~\ref{lemma_liminf} with a certain restriction on $z$.

It was shown in \cite[Proposition 3.8]{DGY} that there exist constants $C, \xi > 0$ such that for any $\omega\in[0,1),\,k\in\N, z\in \sigma_k(\delta), 1\leq n\leq q_k$
\begin{equation}\label{transfer_matrix_upper_bound}
\norm{M(n,0,z,\omega)} \leq Cn^\xi
\end{equation}
So we can estimate the integral \ref{outside_probability_estimate_fibb} from below by restricting it to the set
\begin{equation*}
\set{E\in\R: \, E+i/T\in \sigma_k(\delta)}
\end{equation*}

Let $\supp\psi\subseteq[-s,s]$. To use Lemma~\ref{lemma_liminf}, we want to ensure that for at least one solution $v_{E}$ of $Hv = Ev$,
\begin{equation}\label{large_product_condition}
|\sprod{v_{E|s}}{\psi_{|s}}| \geq \varepsilon_0
\end{equation}
for a rich set of energies $E$. $\varepsilon_0$ will be chosen later in the proof.

Fix any initial conditions for a solution of $Hv = Ev$, e.g. Dirichlet condition: $v(0) = 1, v(1) = 1$. Then $v_E(n)$ as a function of $E$ is a polynomial of degree at most $|n|$. Therefore $\sprod{v_{E|s}}{\psi} = p(E)$ is a polynomial of $E$ of degree at most $s$.
For a given $\varepsilon_0$, define
\begin{equation*}
A(\varepsilon_0) = \set{E\in\R : |p(E)|<\varepsilon_0}
\end{equation*}
which consists of at most $s$ disjoint intervals.

Since we are only interested in integrating over energies $z = E + i/T$ inside $\sigma_k$, it follows from \ref{transfer_matrix_upper_bound} that
\begin{equation*}
\norm{V_{E}(s)} \leq \norm{M(s,0,\omega,E)} \norm{V_E(0)} \leq C s^\xi
\end{equation*}
\begin{equation*}
\norm{V_{E}(-s-1)} \leq \norm{M(-s-1,0,\omega,E)} \norm{V_E(0)} \leq C's^\xi.
\end{equation*}
If for some $E$ \ref{large_product_condition} is satisfied, then by Lemma~\ref{lemma_liminf}
\begin{align*}
\lim\limits_{\delta\rightarrow 0}\inf\set{\norm{U_z(s)}^2 + \norm{U_z(-s-1)}^2: z\in\C\backslash\R, |z-E|<\delta} > \\
>\frac{\varepsilon_0^2}{(\norm{V_{E}(s)} + \norm{V_{E}(-s-1)})^2} \geq \frac{\varepsilon_0^2}{2(\norm{V_{E}(s)}^2 + \norm{V_{E}(-s-1)}^2)} \geq \\
\geq \varepsilon_0^2 C_1 =: \varepsilon_1 .
\end{align*}
For any such $E$ define $\Delta(E)$ such that
\begin{equation}\label{difference_solution_estimate}
\norm{U_z(s)}^2 + \norm{U_z(-s-1)}^2 > \varepsilon_1/2
\end{equation}
for any $z\in\C$ satisfying $|z-E|_{max}<\Delta(E)$, where $|x+iy|_{max} = \max\set{|x|,|y|}$. Working with this norm is convenient because the set $B_E(\Delta(E)) = \set{z: |z-E|_{max}<\Delta(E)}$ is an open rectangle in the complex plane.

Choose $K>0$  such that the spectrum of $H$, $\Sigma$ is contained in $[-K,K]$. Let $I_{\varepsilon_0}$ be a closed interval in $E\in[-K-1,K+1]\backslash A(\varepsilon_0)$ that contains a point of the spectrum in the interior. Here we choose $\varepsilon_0$ sufficiently small so that $I_{\varepsilon_0}$ exists.
\begin{equation*}
\set{B_E(\Delta(E)): \, E\in I_{\varepsilon_0}}
\end{equation*}
 is an open cover of $I_{\varepsilon_0}$. Choose a finite subcover and let $\Delta_0$ be the minimal radius (in the $|\cdot|_{max}$ metric) of the obtained rectangles.

Now we consider $T>2/\Delta_0$ in \ref{outside_probability_estimate_fibb} in order to ensure that \ref{difference_solution_estimate} holds for all $z = E + i/T, E\in I_{\varepsilon_0}$.

In order to use \ref{transfer_matrix_upper_bound} in the integral \ref{outside_probability_estimate_fibb} and avoid the set of $E$'s for which \ref{large_product_condition} is not satisfied, i.e. $A(\varepsilon_0)$, we will restrict the integral to the set $\sigma_k\cap I_{\varepsilon_0}$.

Analogously to \ref{inner_radii} we define
\begin{equation*}
r_k^{\varepsilon_0}(\delta) = \max\limits_{I_{\varepsilon_0}}\set{r_k^{(j)}(\delta)}
\end{equation*}
where maximum is taken over $j\in\set{1,\ldots,q_k}$ satisfying $E_k^{(j)}\in I_{\varepsilon_0}$.

As in the proof of \cite[Proposition 3.8]{DGY} define
\begin{equation*}
 l= \frac{\limsup\limits_{k\rightarrow\infty} \frac{1}{k}\log\frac{1}{r_k^{\varepsilon_0}(\delta)}}{\log\varphi} + \rho
\end{equation*}
Then for some $C_\delta>0$, $C_\delta F_k^l \geq \frac{2}{r_k^{\varepsilon_0}(\delta)}$ for any $k$. Fix some $N$ and let  $ T \geq C_\delta N^l$, which means, in particular, $1/T \leq r_k^{\varepsilon_0}(\delta)/2$.

Finally we can estimate \ref{outside_probability_estimate_fibb}:
\begin{multline}\label{outside_probability_estimate_fibb_2}
\ldots \geq \frac{1}{2\pi T} N \int\limits_{J_{\varepsilon_0}} \left( \max\limits_{N< n \leq 2N} \norm{M(n,s+1,E+\frac{i}{T},\omega)} \right)^{-2}\left(\frac{\varepsilon_1}{2}\right)^2 dE =\\
= \frac{1}{2\pi T} N \int\limits_{J_{\varepsilon_0}} \left( \max\limits_{N< n \leq 2N} \norm{M(n-s-1,0,E+\frac{i}{T},\omega-(s+1)\theta)}\right)^{-2}\left(\frac{\varepsilon_1}{2}\right)^2dE \geq \\
\geq \frac{N}{2\pi T} \left(\frac{\varepsilon_1}{2}\right)^2 C^2 (2N)^{-2\xi} |J_{\varepsilon_0}| \gtrsim \frac{1}{T}N^{1-2\xi}\min\set{{r_k^{\varepsilon_0}(\delta), r_{k+1}^{\varepsilon_0}(\delta)}} \gtrsim \frac{1}{T^2}N^{1-2\xi}.
\end{multline}

Now for any sufficiently large $T$, choose $k$ maximal with $C_\delta F_k^l \leq T$. Then
\begin{equation*}
C_\delta F_k^l \leq T < C_\delta F_{k+1}^l \leq 2^lC_\delta F_k^l.
\end{equation*}
Let $N = T^{1/l-\rho'}$, then for sufficiently large $T$, $T> C_\delta N^l$ and it follows from \ref{outside_probability_estimate_fibb_2} that
\begin{equation*}
P(T^{\frac{1}{l}-\rho'},T) \gtrsim T^{-2 + (\frac{1}{l}-\rho')(1-2\xi)},
\end{equation*}
for any $\rho'>0$ such that $\frac{1}{l}-\rho' > 0$.
Then
\begin{equation*}
S^-(\psi,\frac{1}{l}-\rho') \leq 2 - (\frac{1}{l}-\rho')(1-2\xi) < \infty.
\end{equation*}
Therefore
\begin{equation*}
\alpha^-_u \geq \frac{1}{l}-\rho' = \left(\frac{\limsup\limits_{k\rightarrow\infty} \frac{1}{k}\log\frac{1}{r_k^{\varepsilon_0}(\delta)}}{\log\varphi} + \rho\right)^{-1}-\rho'
\end{equation*}
Since $\rho, \rho'$ are arbitrarily close to 0,
\begin{equation*}
\alpha^-_u \geq \frac{\log\varphi}{\limsup\limits_{k\rightarrow\infty} \frac{1}{k}\log\frac{1}{r_k^{\varepsilon_0}(\delta)}}
\end{equation*}

It was proved in \cite[Proposition 3.8]{DGY} that
\begin{equation*}
\frac{1}{r_k^{(j)}} \leq \frac{(2+3\delta)^2}{(1+\delta)(1+2\delta^2)} \vabs{x_k'(E_k^{(j)})}
\end{equation*}
It follows that
\begin{equation*}
\frac{1}{r_k^{\varepsilon_0}} \leq \frac{(2+3\delta)^2}{(1+\delta)(1+2\delta^2)} \min\limits_{I_{\varepsilon_0}} \vabs{x_k'(E_k^{(j)})}
\end{equation*}
which implies
\begin{equation*}
\alpha^-_u \geq \frac{\log\varphi}{\limsup\limits_{k\rightarrow\infty} \frac{1}{k}\log\left(\min\limits_{I_{\varepsilon_0}} \vabs{x_k'(E_k^{(j)})}\right)}
\end{equation*}

In \cite[Lemma 3.6]{DGY} the authors proved that the limit $\lim\limits_{k\rightarrow\infty} \frac{1}{k}\inf\limits_{p\in \Lambda_\lambda}\log\norm{DTr^k(p)|_{E_p^u}}$ exists and is equal to $\inf\limits_{p\in Per(T_\lambda)} \lyap^u(p)$.
From lemmas \ref{lem_derivative_Tr_1} and \ref{lem_derivative_Tr_2} it follows that
\begin{equation*}
\limsup\limits_{k\rightarrow\infty} \frac{1}{k}\log\left(\min\limits_{I_{\varepsilon_0}} \vabs{x_k'(E_k^{(j)})}\right) = \lim\limits_{k\rightarrow\infty} \frac{1}{k}\inf\limits_{p\in \Lambda_\lambda}\log\norm{DTr^k(p)|_{E_p^u}}
\end{equation*}
Thus the proof is complete.
\end{proof}

\subsection{Sturmian Potentials}
The Fibonacci operator is a special case of a more general and also well studied family of operators defined by \ref{fibbonacci_potential} with $\theta $ being an arbitrary irrational number. Such potentials are called \textit{Sturmian}. The band structure of the spectrum analysis for this model was applied to estimating transport exponents for single-site initial states in \cite{DGLQ14}. The ideas described in the previous section can be used to extend some of their results to initial states with wide support.
\begin{theorem} For Lebesgue almost every frequency $\theta$, uniformly in phase $\omega$ and for any exponentially decaying $\psi$ the following holds:
\begin{equation*}
\limsup\limits_{\lambda\rightarrow\infty} \alpha^+_u(\psi)\cdot \log \lambda \leq \frac{\pi^2}{12\log\frac{1+\sqrt{5}}{2}}
\end{equation*}
\end{theorem}

\begin{remark}
The proof is based on an inequality that is similar to \ref{fibb_exponent_upper_1}, which is proved in \cite[Proposition 4.8,b)]{DGLQ14} and is analogous to the proof of Proposition~\ref{thm_fibb_upper}.
\end{remark}

\subsection{H\"older Continuous Quasi-periodic Potentials}
A dynamically defined Schr\"odinger operator is called \textit{quasi-periodic} if the underlying dynamical system is a minimal rotation on the torus $\mathbb{T}^m$. We will focus on the case when $m = 1$, in which case minimal rotations are generated by irrational frequencies. Regularity of the sampling function is the key to the properties of the resulting operator and in particular quantum dynamics.

 Vanishing transport exponents were proved in \cite{JM16} for a class of piecewise H\"older continuous sampling functions under the assumption of positivity and continuity of the Lyapunov exponent.  Using their estimates on the norms of transfer matrices we were able to extend their result \cite[Theorem 1.7]{JM16} to the case of exponentially decaying initial state.

We say that $\theta$ satisfies the Diophantine condition $D_2$ if for some $\kappa>0$
\begin{equation*}
q_{n+1} < q_n^{1+\kappa} \text{ for all large }n.
\end{equation*}
We will need the following notation:
\begin{equation*}
\Phi_{\alpha, m}(E,T) = \inf \frac{1}{T^m}\max\limits_{0\leq n \leq T^\alpha}\norm{M(n,0,z,\omega)}
\end{equation*}
where the infinum is taken over $|z-E|\leq T^{-\alpha}$ and all $\omega\in[0,1)$.
\begin{lemma}[Proposition 2.1 in \cite{JM16}]\label{lemma_JM}
Let $f$ be piecewise H\"older continuous, $\gamma_\theta(E) > \Gamma > 0$ for every $E$ in a compact Borel set $U\subset \R$ and is continuous on $U$. Then for any $0 < \alpha \leq 1, m > 1$, there is $c > 0$ and a sequence of times $(T_n)$ so that for every $E\in U, n\geq 1$
\begin{equation}\label{JM_cond}
\Phi_{\alpha,m}(E,T_n) > c.
\end{equation}
If $\theta$ is $D_2$-Diophantine, then there is $T_0 > 0$ such that \ref{JM_cond} holds for all $T>T_0$.
\end{lemma}

\begin{theorem}\label{thm_JM}
Let the potential be generated by piecewise H\"older $f$ and , $\theta \in \R\backslash \Q$, $\psi$ is exponentially decaying. Assume that the Lyapunov exponent, L(E), is positive for every $E\in \R$ and is continuous on $\R$, then
\begin{equation*}
\alpha_u^-(\theta, \omega) = 0 \,\text{ for any }\, \omega\in[0,1).
\end{equation*}
If in addition $\theta$ satisfies condition $D_2$, then
\begin{equation*}
\alpha_u^+(\theta, \omega) = 0 \,\text{ for any }\, \omega\in[0,1).
\end{equation*}
\end{theorem}
\begin{proof}
Follows from Lemma~\ref{lemma_JM} and Corollary~\ref{cor_dyn_upper_bound}.
\end{proof}

Positivity \cite{sorets1991} and continuity \cite{BJ02} of the Lyapunov exponent are known for quasi-periodic potentials with analytic sampling function. So we can apply Theorem~\ref{thm_JM}.
\begin{corollary}
If $W(n) = f(n\theta+\omega)$, where $f$ is analytic on $\mathbb{T}$, then the conclusion of Theorem~\ref{thm_JM} holds true.
\end{corollary}

In \cite{DT07} the same result was proved for trigonometric polynomial sampling function $f$ under a weaker Diophantine condition:
\begin{equation}\label{diophantine_condition}
\lim\limits_{k\rightarrow \infty} \frac{\log q_{k+1}}{q_k} = 0
\end{equation}

It also holds true for exponentially decaying initial states.

\subsection{Period Doubling Substitution}
For the period doubling model linear growth of transfer matrices norms for certain energies was shown in \cite{DT03} and can be applied in our situation.

We define the period doubling Hamiltonian as follows. Consider the following substitution $S$ on the alphabet $A = \set{0,1}$: $S(0) = 01, S(1) = 00$. There is a unique one-sided sequence $\omega^+ = 0100010\ldots$ which starts with $0$ and is fixed by $S$. If $L:A^\Z\rightarrow A^\Z$ is the left shift and $\omega^0 = \ldots000\omega^+\in A^\Z$, then define $\Omega_{pd}$ to be the set of limit points of $Orb_L^+(\omega^0)$.

Given $\lambda > 0, \omega = (\ldots \omega_{-1}\omega_0\omega_1\omega_2\ldots)\in\Omega_{pd}$ we define the potential $W_{\lambda,\omega}(n)$ by $W_{\lambda,\omega}(n) = \lambda \omega_n$.

\begin{lemma}\cite[Proposition 4.2]{DT03}\label{lemma_pd_transfer_matrix}
For any $\lambda > 0$ there are countably many energies $E_0$ such that
\begin{equation*}
\norm{M(n,m,E,\omega)} \leq C_1(\lambda,E)(\sqrt{2} + C_2(\lambda,E)|n-m|)
\end{equation*}
where $C_1, C_2$ are constants.
\end{lemma}

\begin{theorem}
Let $\lambda > 0, \omega\in\Omega_{pd}, \psi$  has finite support. Then $\beta^-(\psi,p) > \frac{1}{2}-\frac{5}{2p}$.
\end{theorem}
\begin{proof}
Lemma \ref{lemma_pd_transfer_matrix} allows us to apply Proposition~\ref{prop_one_energy_inf_many} with $\alpha = 1$, thus proving the estimate.
\end{proof}

\subsection{Thue-Morse Substitution}
The Thue-Morse Hamiltonian is defined in a similar way as the period doubling via the substitution $S(0) = 01, S(1) = 10$. The corresponding space of sequences we denote by $\Omega_{tm}$.
\begin{theorem}
Let $\lambda > 0, \omega\in\Omega_{tm}, \psi$  has finite support. Then $\beta^-(\psi,p) > 1-\frac{1}{p}$.
\begin{proof}
In \cite[Proposition 5.1]{DT03} it was shown that for an infinite set of energies transfer matrices are bounded. Thus Proposition~\ref{prop_one_energy_inf_many} is applicable with $\alpha = 0$.
\end{proof}
\end{theorem}
\begin{corollary}
Under the assumptions of the previous theorem $\alpha^-_u(\psi) = \alpha^+_u(\psi) = 1$.
\end{corollary}

\subsection{Random Polymer Model}\label{random polymer model}
For a random polymer model with critical energy it was proved in \cite{JSS03} and \cite{JS-B07} that
\begin{equation*}
\beta^-(\delta_0,p) = \max\set{0,1-\frac{1}{2p}}
\end{equation*}
We prove the lower bound for a wide class of compactly supported states.

To construct a potential of the random polymer model one starts with a biinfinite sequence $\omega = (\omega_l)_{l\in\Z} \in \{+,-\}^\Z$.  Let $\hat{W}_\pm = (\hat{W}_\pm(0),\ldots,\hat{W}_\pm(L_\pm-1))$ be a pair of finite sequences of real numbers. These sequences are called potentials of two polymers. The whole potential is constructed as a concatenation : $W_\omega = (\cdots\hat{W}_{\omega_0}\hat{W}_{\omega_0}\cdots)$. The space $\{+,-\}^\Z$ is endowed with probability measure $\mathbf{P}$. It is a product measure of Bernoulli measures on $\{+,-\}$. One can interpret this sequence as a sequence of i.i.d. Bernoulli random variables. To define the probability space $(\Omega,\mathbf{P})$ properly we need to add an extra coordinate to indicate the origin inside $\hat{W}_{\omega_0}$. The precise definition can be found in \cite[Section 4]{JSS03}.

An energy $E_c$ is called critical if the transfer matrices over two polymers have traces less then 2 and commute at this energy or are equal to $\pm Id$. It is readily seen that all transfer matrices are bounded at a critical energy. So we can apply Corollary~\ref{cor_one_energy_finite} with $\alpha = 0$ to get

\begin{theorem}
Let $H_\omega$ be a random polymer model with critical energy $E_c$. For any $\omega\in\Omega$ and any $\psi$ with finite support if $\langle v_{E_c},\psi\rangle \neq 0$ with $v_{E_c}$ being some solution of $H_\omega v = E_cv$, then
\begin{equation*}
\beta^-(\psi,p) \geq 1 - \frac{1}{p}
\end{equation*}
\end{theorem}

However the authors of \cite{JSS03} showed a stronger statement: almost surely the transfer matrices remain bounded in some interval of energies on a suitable length scale. In particular, they proved

\begin{lemma}[Theorem 7 in \cite{JSS03}]
Let $|\langle e^{2i\eta_\pm}\rangle| \neq 1$\footnote{$\eta_\pm$ are defined as the rotation numbers of the corresponding one polymer transfer matrices.} and $\epsilon > 0$. Then there are $0 < c,c',C < \infty$ such that for any $N\in\N$ there exists a set $\Omega_N(\epsilon)\subset \Omega$ satisfying $\mathbf{P}(\Omega_N(\epsilon)) = \mathcal{O}(e^{-cN^\epsilon})$ and such that for any $\omega \in \Omega\backslash\Omega_N(\epsilon)$ one has
\begin{equation}\label{polymer_transfer_matr_estimate}
\norm{M(n,m,E_c+\delta+is,\omega)} < C
\end{equation}
for all $0\leq m\leq n \leq N$, $|\delta| \leq N^{-1/2-\epsilon}$, $|s| < c'/N$.
\end{lemma}


This lemma allows us to prove the main result of this section.
\begin{theorem}\label{thm_polymer_lower_bound}
Let $H_\omega$ be a random polymer model with critical energy $E_c$ and  \\$|\langle e^{2i\eta_\pm}\rangle| \neq 1$. There is a full measure set $\tilde\Omega\subset\Omega$ such that for any $\omega\in\tilde\Omega$ and any finitely supported initial state $\psi$, satisfying $\langle v_{E_c},\psi\rangle \neq 0$ with $v_{E_c}$ being some solution of $H_\omega v = E_cv$,
\begin{equation*}
\beta^-(\psi,p) \geq 1 - \frac{1}{2p}
\end{equation*}
\begin{proof}
Define $\tilde\Omega = \bigcup\limits_{N = 1}^\infty \bigcap\limits_{K\geq N}^\infty \Omega\backslash\Omega_K(\epsilon)$. By the Borel-Cantelli lemma $\mathbf{P}(\tilde\Omega) = 1$ and for every $\omega\in\tilde\Omega$ \ref{polymer_transfer_matr_estimate} holds for all $N>N(\omega)$.

Let $\omega\in\tilde\Omega$ and $\epsilon > 0$. We want to use the set $(E_c - N^{-1/2-\epsilon},E_c + N^{-1/2-\epsilon})$ as $A(N)$ in the theorem \ref{main_result}. Since the polynomial $p(E) = \sprod{v_E}{\psi}$, where $v_E $ is the solution to the eigenvalue equation with the same initial conditions as $v_{E_c}$ in the assumption, is non-zero at $E_c$, there is a neighborhood of $E_c$ where $|p(E)|> \varepsilon_0 > 0$. Denote this neighborhood by $U(E_c)$. To ensure that conditions  \ref{transfer_matrix_bound} and \ref{positive_product_condition} are satisfied for all energies in $A(N)$, pick $N_0$ large enough so that $N_0 > N(\omega)$ and $A(N_0)\subset U(E_c)$.

Now the conclusion \ref{moments_bound} of the Theorem~\ref{main_result} holds for all sufficiently large $T$ with $\alpha = 0$ and
\begin{equation*}
|B(T)| = |A(T)|+\frac{2}{T} \geq C_1 T^{-\frac{1}{2}-\epsilon}
\end{equation*}
with some universal constant $C_1$.
Then from the estimate \ref{moments_bound} it follows that for large enough $T$,
\begin{equation*}
\langle|X|^p_\psi \rangle(T) \geq C_2 T^{p-\frac{1}{2}-\epsilon}
\end{equation*}
Thus $\beta^-(\psi,p) \geq 1-\frac{1}{2p}-\frac{\epsilon}{p}$ and since $\epsilon$ is arbitrary, the proof is complete.

\end{proof}
\end{theorem}

\begin{remark}
The condition $\sprod{v_{E_c}}{\psi} \neq 0$ is generic in $\psi$ in the following sense. Let's restrict our attention to the set  $\set{\psi|\supp\psi = [a,b]} \cong \R^{b-a+1}$. Since the vector space of solutions to $H_\omega v = E v$ is two dimensional and every solution is determined by the initial values $V(0)$, then the nonvanishing of the inner product condition is not satisfied for $\psi$ satisfying two equalities :
\begin{align}
\begin{split}\label{condition_zero_product}
&\sprod{v_{E_c}^1}{\psi} = 0 \\
&\sprod{v_{E_c}^2}{\psi} = 0
\end{split}
\end{align}
where $v_{E_c}^1$ and $v_{E_c}^2$ are two solutions of $H_\omega v = E_c v$ with initial conditions $V^1(a) = \colvec{2}{0}{1}$ and $V^2(a) = \colvec{2}{1}{0}$. Clearly $v^1_{|[a,b]}$ and $v^2_{|[a,b]}$ are not collinear, so condition \ref{condition_zero_product} defines a subspace of codimension 2 in $\R^{b-a+1}$.
So Theorem~\ref{thm_polymer_lower_bound} can be reformulated for almost every $\psi$ supported on $[a,b]$.
\end{remark}
\begin{remark}
In the special case when $|\supp\psi| = 2$, the assumption of  Theorem~\ref{thm_polymer_lower_bound} is satisfied for any $\psi$.
\end{remark}
\begin{remark}
The set of exceptional states satisfying \ref{condition_zero_product} depends on $\omega$, however note that only a finite segment of the potential, $\omega_{|[a,b]}$, defines this condition. From the construction of $\omega$ from two fixed sequences it follows that there are finitely many different patterns of length $b-a+1$ that may appear in $\omega$. Call these patterns $\omega^1,\ldots,\omega^K\in \R^{b-a+1}$. Then the complement of the union of $K$ codimension 2 subspaces defined by \ref{condition_zero_product} consists of vectors for which nonvanishing condition from the theorem is satisfied.

Thus we can reformulate the theorem in the following way:
\end{remark}
\begin{corollary}
Let $H_\omega$ be a random polymer model with critical energy $E_c$ and \\ $|\langle e^{2i\eta_\pm}\rangle| \neq 1$. Fix some $a,b\in\Z$. For almost every $\psi$ with $\supp\psi = [a,b]$ and almost every $\omega\in\Omega$
\begin{equation*}
\beta^-(\psi,p) \geq 1 - \frac{1}{2p}.
\end{equation*}
\end{corollary}

\bibliography{BiblBase}
\bibliographystyle{plain}

\end{document}